\documentclass[sn-mathphys-num]{sn-jnl}

\usepackage{graphicx,color}
\usepackage{amssymb}
\usepackage{amsmath,amsthm}
\usepackage{enumerate}
\usepackage{subcaption}
\usepackage{tikz}
\usetikzlibrary{arrows,calc,patterns}

\usepackage{diagbox}

\def\thanks#1{{\let\thefootnote\relax\footnote{#1.}\setcounter{footnote}{0}}}

\DeclareMathOperator{\aut}{Aut} 
\DeclareMathOperator{\cone}{cone} 
\DeclareMathOperator{\inte}{int} 
\DeclareMathOperator{\bd}{bd} 
\DeclareMathOperator{\interior}{int} 
\DeclareMathOperator{\linspan}{span} 
\DeclareMathOperator{\vol}{vol} 

\newcommand\del{\backslash} 
\newcommand{\R}{\mathbb{R}} 
\newcommand{\Z}{\mathbb{Z}} 
\newcommand{\scalPr}[2]{\langle{#1},{#2}\rangle}
\newcommand\norm[1]{\left\lVert#1\right\rVert}
\newcommand\inprod[1]{{\left\langle #1 \right\rangle}}
\newcommand\ov{\overline}

\newtheorem{thm}{Theorem}

\newtheorem{lemma}[thm]{Lemma}
\newtheorem{cor}[thm]{Corollary}
\newtheorem{rem}[thm]{Remark}
\newtheorem{example}[thm]{Example}

\begin{document}

\title[Normalizations of cone factorizations]{Normalizations of factorizations over convex cones and their effects on extension complexity}

\author[1]{\fnm{Adam} \sur{Brown}}\email{ajmbrown@gatech.edu}
\equalcont{These authors contributed equally to this work.}

\author*[2]{\fnm{Kanstantsin} \sur{Pashkovich} \thanks{January 6, 2025 (revised: \today)\\ Research of K. Pashkovich
was supported in part by Discovery Grants from the Natural Sciences and Engineering Research
Council (NSERC) of Canada}}\email{kpashkov@uwaterloo.ca}
\equalcont{These authors contributed equally to this work.}

\author[2]{\fnm{Levent} \sur{Tun\c{c}el} \thanks{Research of L. Tun\c{c}el
was supported in part by Discovery Grants from the Natural Sciences and Engineering Research
Council (NSERC) of Canada}}\email{levent.tuncel@uwaterloo.ca}
\equalcont{These authors contributed equally to this work.}

\affil[1]{\orgdiv{School of Mathematics}, \orgname{Georgia Institute of Technology}, \orgaddress{\street{686 Cherry Street}, \city{Atlanta}, \postcode{30318}, \state{Georgia}, \country{United States}}}

\affil[2]{\orgdiv{Department for Combinatorics and Optimization}, \orgname{University of Waterloo}, \orgaddress{\street{200 University Avenue West}, \city{Waterloo}, \postcode{N2L 3G1}, \state{Ontario}, \country{Canada}}}

\abstract{Factorizations over cones and their duals play central roles for many areas of mathematics and computer science. One of the reasons behind this is the ability to find a representation for  various objects using a well-structured family of cones, where the representation is captured by the factorizations over these cones.   Several major questions about factorizations over cones remain open even for such well-structured families of cones as non-negative orthants and positive semidefinite cones. Having said that, we possess a far better understanding of factorizations over  non-negative orthants and positive semidefinite cones than over other families of  cones. One of the key properties that led to this better understanding is the ability to normalize factorizations, i.e., to guarantee that the norms of the vectors involved in the factorizations are bounded in terms of an input and in terms of  a constant dependent on the given cone. Our work aims at understanding which cones guarantee that factorizations over them can be normalized, and how this effects extension complexity of polytopes over such cones.}


\maketitle

\section{Introduction}

A convex cone is called \emph{regular}, if it is closed, pointed and has nonempty interior. 
Let $\mathcal C$ be a regular convex cone in $\R^n$, then its \emph{dual} $\mathcal C^\star$ is defined as
\[
    \mathcal C^\star:=\{x\in \R^n\,:\,\langle x,\, y\rangle \geq 0\quad \text{for all} \quad y\in \mathcal C\}\,.
\]
If $\mathcal{C}$ is a regular convex cone then so is its dual cone.
\emph{Factorization problem over a cone} (and its dual) is described
in~\cite[Definition 2.2]{Gouveia}. 
Factorizations over cones play an important role for several areas of mathematics and computer science~\cite{lee1999}, \cite{devarajan2008}, \cite{ding2005}, \cite{kubjas2015},\cite{fawzi2015}, etc. Intuitively, non-negative factorizations (and the related nonnegative rank) allow to compress information about features of various objects to fewer key features. In this spirit, Yannakakis established the connection between the nonnegative rank of slack matrices and ``compressing" polytopes. In particular, the work of Yannakakis~\cite{yannakakis} showed that so-called \emph{linear extension complexity} of a nontrivial polytope equals the \emph{nonnegative rank} of its slack matrix.

 Now we know explicit families of $0/1$-polytopes that do not admit any linear or positive semidefinite extension of polynomial size~\cite{Fiorini12}, \cite{rothvoss2017} and \cite{lee2015}.  
Surprisingly, for several decades after the work of Yannakakis~\cite{yannakakis}, it was an open question  whether every $0/1$-polytope in $\R^n$ has linear extension complexity that is polynomial in $n$. For the first time this question was answered in negative by Rothvoss~\cite{rothvoss201101}. The arguments in the work of Rothvoss~\cite{rothvoss201101} are based on counting and  can be starightforwardly extended to stronger collections of extensions, as long as these collections admit \emph{a normalization of factorizations}.

In their work on positive semidefinite cones, Bri\"{e}t et al.~\cite{briet2013existence} provided a negative answer for a broader class of extensions known as positive semidefinite extensions by introducing a normalization approach. Subsequent research by Fawzi et al.~\cite{fawzi2015} expanded on this by presenting another normalization method using John-L\"{o}wner ellipsoids. Averkov et al.~\cite{Averkov2018} similarly employed these ellipsoids to analyze the semidefinite extension complexity of polytopes.

The utility of John-L\"{o}wner ellipsoids for normalization has proven versatile, finding applications beyond the analysis of positive semidefinite cone extension complexity, such as in Banach space theory, as mentioned in the work~\cite[Lemma 4.2]{LMSS2007}. Building on these foundations, our research aims to extend normalization results to more general convex cones, seeking to broaden our understanding of cone normalizations.

Interestingly, while our arguments in Section~\ref{sec:3} may initially appear very distinct, they nonetheless share underlying connections to the problem of constructing structured minimum volume ellipsoids that contain specific sets.

The limitations of extended formulations were studied also in the context of polytopes that are not necessarily 0/1. The normalization of factorizations plays an important role for providing lower bounds on extension complexity  for various families of polytopes. For example, the extension complexity of even very small-dimensional polytopes is not fully understood at the moment. In particular, for polytopes of dimension two, first non-trivial lower bounds for linear extension complexity were provided in~\cite{Fiorini12}: a lower bound of $\Omega(\sqrt n)$ for a general $n$-gon and $\Omega(\sqrt n/\log(n))$ for an integral $n$-gon, where the last result was building on the normalization and counting ideas from~\cite{rothvoss201101}. There is a gap between these lower bounds and currently best upper bounds on extensions complexity of polygons. Recently, Shitov\cite{shitov2020} provided a sublinear upper bound of $O(n^{2/3})$ for linear extension complexity of general $n$-gons. In some randomized constructions for polytopes with $n$ vertices and with a  constant dimension, in~\cite{kwan2022} it was shown that the linear extension complexity is asymptotically almost surely equal to $\Theta(\sqrt{n})$. Thus, even for polytopes of dimension two, many questions about linear extension complexity remain open. 

\subsection{Problem of Normalizing Factorizations}\label{sec:problem_definition}

Before further discussion, let us formulate the question about normalization of convex cones.
Given a regular convex cone $\mathcal C$ we say that it is possible to \emph{normalize} factorizations over this cone and its dual with a constant $f_{\mathcal C}\in \R_+$ if  for every pair of compact sets 
 $\mathcal A\subseteq \mathcal C$ and $\mathcal B\subseteq \mathcal C^\star$ there exists a pair of sets $\widetilde{\mathcal A}\subseteq \mathcal C$ and $\widetilde{\mathcal B}\subseteq \mathcal C^\star$ and two bijections $g:\mathcal A\rightarrow \widetilde{\mathcal A}$, $q:\mathcal B \rightarrow\widetilde{\mathcal B}$ such that 
\begin{equation}\label{eq:normalization_matrix}
\langle a,\, b\rangle= \langle g(a),\, q(b)\rangle\qquad\text{for all}\quad a\in \mathcal A,\, b\in \mathcal B
\end{equation}
\begin{align}
    &\|\widetilde a\|_2\leq f_{\mathcal C}\sup_{a\in \mathcal A, b\in \mathcal B} \langle a,\, b\rangle \qquad\text{for all} \quad \widetilde a \in \widetilde{\mathcal A}\label{eq:normalization_first_set}\\
    &\| \widetilde b\|_2\leq f_{\mathcal C}\sup_{a\in \mathcal A, b\in \mathcal B} \langle a,\, b\rangle \qquad\text{for all}\quad \widetilde b \in \widetilde{\mathcal B}\label{eq:normalization_second_set}\,.
\end{align}
Informally, we will also say that in the above case the sets $\widetilde{\mathcal A}$ and $\widetilde{\mathcal B}$ are being \emph{rescaled}.

In Section 2, we prove that if a regular convex cone admits normalization of factorizations then such a normalization can be achieved via automorphisms of the cone. In Section 3, for indecomposable symmetric cones $\mathcal{C}$, we relate the normalization constant $f_{\mathcal C}$ to the Carath\'{e}odory number of $\mathcal{C}$. More generally, for indecomposable homogeneous cones ${\mathcal C}$, we bound the normalization constant $f_{\mathcal C}$ in terms of the barrier parameter of certain self-concordant barriers for $\mathcal C$. While we prove that normalization of factorizations is possible over every homogeneous cone, we show that normalization of factorizations is impossible for even some elementary hyperbolic cones. In Section 4, we discuss the consequences of normalization of factorizations on extension complexity.

\section{Normalization and Automorphisms}
In this section, we show that if a regular convex cone allows a normalization of factorizations, then there exists a normalization induced by automorphisms of this cone and its dual. First, we show that if we are rescaling sets $\mathcal A\subseteq \mathcal C$, $\mathcal B\subseteq \mathcal C^\star$ which span $\R^n$, then the bijections $g:\mathcal A\rightarrow \widetilde{\mathcal A}$, $q:\mathcal B \rightarrow\widetilde{\mathcal B}$ as in Section~\ref{sec:problem_definition} will always be invertible linear maps.

\begin{lemma}\label{linear-normalization}
    Let $\mathcal C$ be a regular convex cone and let $\mathcal C^*$ be its dual. Let $\mathcal A$ and $\mathcal B$ be subsets of $\mathcal C$ and $\mathcal C^*$ respectively. If $\mathcal A$ and $\mathcal B$ both span $\R^n$ and there exist two sets $\widetilde{\mathcal A}\subseteq \mathcal C$ and $\widetilde{\mathcal B}\subseteq \mathcal C^\star$ and two bijections $g:\mathcal A\rightarrow \widetilde{\mathcal A}$, $q:\mathcal B \rightarrow\widetilde{\mathcal B}$ such that 
    $$\langle a,\, b\rangle= \langle g(a),\, q(b)\rangle\qquad\text{for all}\quad a\in \mathcal A,\, b\in \mathcal B,$$
    then $g$ and $q$ are invertible linear maps. Moreover, $q$ is the inverse-adjoint of $g$.
\end{lemma}
\begin{proof}
Let us select $n$ linearly independent vectors $a_1$, $a_2$, \ldots, $a_n$ in~$\mathcal A$ and $n$ linearly independent vectors $b_1$, $b_2$, \ldots, $b_n$ in~$\mathcal B$. Let us construct matrices $R$, $\widetilde R$, $L$, $\widetilde L\in \R^{n\times n}$, where for each $i \in \{1, 2, \ldots,n\}$ the matrices $R$, $\widetilde R$, $L$ and $\widetilde L$ have as $i$-th row the vector $a_i^{\top}$, $g(a_i)^{\top}$, $b_i^{\top}$ and $q(b_i)^{\top}$, respectively. Clearly, the matrices $R$ and $L$ are non-singular. The matrices $\widetilde R$ and $\widetilde L$ are also non-singular, since $R\cdot L^{\top}$ is a non-singular matrix and $R \cdot L^{\top}=\widetilde R \cdot \widetilde L^{\top}$. 

For all $a\in \mathcal A$ we have $L\cdot a=\widetilde L \cdot g(a)$, so for all $a\in \mathcal A$ we can write $g(a)=\widetilde{L}^{-1}\cdot L\cdot a$, showing that $g$ is a linear map. Analogously, for all $b\in \mathcal B$ we have $q(b)=\widetilde{R}^{-1}\cdot R\cdot b$, showing that $q$ is a linear map.

From the formula $R \cdot L^{\top}=\widetilde R \cdot \widetilde L^{\top}$, we can rearrange to get 
$$R^{-1}\cdot \widetilde R = L^{\top}\cdot (\widetilde L)^{-\top} = \left( L^{-1} \cdot \widetilde L\right)^{-\top}.$$
Since we noted that $g(a)=\widetilde{L}^{-1}\cdot L\cdot a$ and $q(b)=\widetilde{R}^{-1}\cdot R\cdot b$, we conclude that $q$ is the inverse-adjoint of $g$.
\end{proof}

Now we show that if any compact subset of a given cone can be normalized, then they can be normalized by automorphisms of the cone itself.

\begin{thm}\label{thm:rescaling_automorphism}
  Let $\mathcal C\subseteq \R^n$ be a regular convex cone. Suppose there exists a constant $f_{\mathcal C}\in \R_+$ such that for every pair of compact sets 
 $\mathcal A\subseteq \mathcal C$ and $\mathcal B\subseteq \mathcal C^\star$, there exist a pair of sets $\widetilde{\mathcal A}\subseteq \mathcal C$ and $\widetilde{\mathcal B}\subseteq \mathcal C^\star$ and a pair of bijections $g:\mathcal A\rightarrow \widetilde{\mathcal A}$, $q:\mathcal B \rightarrow\widetilde{\mathcal B}$ such that \eqref{eq:normalization_matrix},~\eqref{eq:normalization_first_set},~\eqref{eq:normalization_second_set} hold. Then, there exist bijections $g:\mathcal A\rightarrow \widetilde{\mathcal A}$, $q:\mathcal B \rightarrow\widetilde{\mathcal B}$ such that \eqref{eq:normalization_matrix},~\eqref{eq:normalization_first_set},~\eqref{eq:normalization_second_set} hold and $g$ is an automorphism of $\mathcal C$ and $q$ is an automorphism of $\mathcal{C}^\star.$
 
\end{thm}

Before we start the proof of Theorem~\ref{thm:rescaling_automorphism}, let us outline the idea behind it. In the proof, we construct $\mathcal A'\subseteq \mathcal C$ and $\mathcal B'\subseteq \mathcal C^\star$ such that $\mathcal A\subseteq \mathcal A'$, $\mathcal B\subseteq \mathcal B'$ and 
\[
    \sup_{a\in \mathcal A, b\in \mathcal B} \langle a,\, b\rangle=\sup_{a'\in \mathcal A', b'\in \mathcal B'} \langle a',\, b'\rangle\,.
\]
Since we assume that it is possible to normalize factorizations over $\mathcal C$ and $\mathcal C^\star$ with respect to the constant $f_{\mathcal C}$, there will exist a pair of sets $\widetilde{\mathcal A}'\subseteq \mathcal C$ and $\widetilde{\mathcal B}'\subseteq \mathcal C^\star$ and  bijections $g':\mathcal A'\rightarrow \widetilde{\mathcal A}'$, $q':\mathcal B' \rightarrow\widetilde{\mathcal B}'$ such that such that the analogues of \eqref{eq:normalization_matrix},~\eqref{eq:normalization_first_set},~\eqref{eq:normalization_second_set} hold for these sets. Clearly, then the restriction of $g'$ to $\mathcal A$ and the restriction of $q'$ to $\mathcal B$ can be used to define a normalization for a factorization corresponding to $\mathcal A$ and $\mathcal B$. The structure of the constructed $\mathcal A'$ and $\mathcal B'$ forces the bijections $g'$ and $q'$ to be an automorphism of $\mathcal C$ and an automorphism of $\mathcal C^\star$, respectively.

\begin{proof}[Proof of Theorem~\ref{thm:rescaling_automorphism}]
First of all, let us consider a special case, when $\sup_{a\in \mathcal A, b\in \mathcal B} \langle a,\, b\rangle$ equals $0$. In this case, both $g$ and $q$ can be defined as $g:x\mapsto \varepsilon\cdot x$ and $q:x\mapsto \varepsilon\cdot x$ for a sufficiently small $\varepsilon$, $\varepsilon>0$, because both $\mathcal A$ and $\mathcal B$ are compact and so bounded. Clearly, for every cone such maps $g$ and $q$ are automorphisms. In the rest of the proof, we assume that $\gamma:=\sup_{a\in \mathcal A, b\in \mathcal B} \langle a,\, b\rangle$ is strictly larger than~$0$.

Since $\mathcal A$ and $\mathcal B$ are compact sets, there exists $M$, $M>0$ such that 
\[
    \|a\|_2\leq M\qquad \text{and}\qquad \|b\|_2\leq M \qquad\text{for all}\quad a\in \mathcal A,\, b\in \mathcal B
\,.\]
Now let us define $\mathcal A'\subseteq \mathcal C$ and $\mathcal B'\subseteq \mathcal C^\star$
\begin{align*}
    &\mathcal A':=\mathcal A\cup \left\{ x\in \mathcal C\,:\,\|x\|_2\leq \min\left\{\frac{\gamma}{M},\sqrt{\gamma}\right\} \right\}\\
    &\mathcal B':=\mathcal B\cup \left\{ x\in \mathcal C^\star\,:\,\|x\|_2\leq \min\left\{\frac{\gamma}{M},\sqrt{\gamma}\right\} \right\}\,.
\end{align*}
It is straightforward to check that both $\mathcal A'$ and $\mathcal B'$ are compact sets and that $\sup_{a\in \mathcal A, b\in \mathcal B} \langle a,\, b\rangle$ equals $\sup_{a'\in \mathcal A', b'\in \mathcal B'} \langle a',\, b'\rangle$.

By assumption, there exist bijections $g':\mathcal{A}'\rightarrow\widetilde{\mathcal A}'$ and $q':\mathcal{B}'\rightarrow\widetilde{\mathcal B}'$ satisfying the analogues of \eqref{eq:normalization_matrix},~\eqref{eq:normalization_first_set},~\eqref{eq:normalization_second_set}. Since $\mathcal C$ is full dimensional the sets $\mathcal{A}'$ and $\mathcal{B}'$ are spanning, so Lemma \ref{linear-normalization} implies that $g'$ and $q'$ are invertible linear maps.

Now let us show that the bijections $g'$ and $q'$ are automorphisms of $\mathcal C$ and $\mathcal C^\star$, respectively. For this, it is enough to show $g'(\cone(\mathcal A'))= \mathcal C$ and $q'(\cone(\mathcal B'))= \mathcal C^\star$, since $\cone(\mathcal A')=\mathcal C$ and $\cone(\mathcal B')=\mathcal C^\star$.  By definition, we have $g'(\cone(\mathcal A'))\subseteq \mathcal C$ and $q'(\cone(\mathcal B'))\subseteq \mathcal C^\star$. So it is enough to show that $g'(\cone(\mathcal A'))\supseteq \mathcal C$ and $q'(\cone(\mathcal B'))\supseteq \mathcal C^\star$. Let us assume, that we have $g'(\cone(\mathcal A'))\supsetneq \mathcal C$, i.e. let us assume that  there exists $\widetilde a' \in \mathcal C$ such that $\widetilde a'$ is not in $g'(\cone(\mathcal A'))$.  The preimage $(g')^{-1}(\widetilde a')$ has a nonnegative scalar product with all vectors $b'\in \mathcal B'$. Since $\cone(\mathcal B')=\mathcal C^\star$, we have that $(g')^{-1}(\widetilde a')$ is in $\mathcal C$. However, since $\widetilde a'$ is not in $g'(\cone(\mathcal A'))$, we know that $(g')^{-1}(\widetilde a')$ is not in $\cone(\mathcal A')=\mathcal C$, contradiction. Analogously, we have $q'(\cone(\mathcal B'))\supseteq \mathcal C^\star$, finishing the proof.
\end{proof}

The next remark is due to the fact that to normalize a factorization over a Cartesian product of two cones, we can normalize the factorizations over each of these cones separately.

\begin{rem}
    Let us be given two regular convex cones $\mathcal C$ and $\mathcal D$, such that it is possible to normalize factorizations with respect to them with constants $f_{\mathcal C}$ and $f_{\mathcal D}$, respectively. Then, for the cone $\mathcal C \times \mathcal D$ it is possible to normalize factorizations with a constant $\max (f_{\mathcal C},  f_{\mathcal D})$. 
\end{rem}

\section{Normalization of Cones}
\label{sec:3}

In our normalization of the cones, we will construct the scaling automorphisms by utilizing a convex optimization
approach via some special classes of strictly convex functions whose domains are the interiors of regular convex cones
in consideration. We start with the so-called $\vartheta$-\emph{logarithmically homogeneous self-concordant barrier (LHSCB)} $F$
(see, Definition 2.3.2 in \cite{NN1994}). Given a function $f : \mathbb{R}^n \to \mathbb{R} \cup \{+\infty\}$, we define the \emph{Legendre-Fenchel conjugate} of $f$,
$f_*: \mathbb{R}^n \to \mathbb{R} \cup \{+\infty\}$
as
\[
f_*(v):=\sup_{x \in \mathbb{R}^n}\left\{-\scalPr{v}{x} -f(x)\right\}.
\]
In the next lemma, we collect some of the properties of $\vartheta$-LHSCB. 

\begin{lemma}
\label{lem:3.0}
Every regular convex cone $\mathcal{C} \subset \mathbb{R}^n$ admits a $\vartheta$-LHSCB $F$, where $1 \leq \vartheta \leq n$.
Moreover, every $\vartheta$-LHSCB $F$ for $\mathcal{C}$ satisfies the following:\\
\begin{enumerate}[(a)]
\item
\label{3.1.0}
For every sequence $\left\{x^{(k)}\right\} \subset \inte(\mathcal{C})$ such that $x^{(k)} \to \bar{x} \in \bd(\mathcal{C})$, we have
\[
F\left(x^{(k)}\right) \to +\infty \,\,\,\, \hspace{1cm} \textup{ and }  \hspace{1cm} \,\,\,\, \norm{\nabla F\left(x^{(k)}\right)} \to +\infty;
\]
\item
\label{3.1.1}
$\scalPr{-\nabla F(x)}{h}^2  \leq \vartheta \scalPr{\nabla^2 F(x)h}{h},$ for all  $x \in \inte(\mathcal{C})$ and for all $ h \in \R^n$;\\
\item
\label{3.1.2}
$\nabla^2 F(x) x = -\nabla F(x),$ for all $ x \in \inte(\mathcal{C})$;\\
\item
\label{3.1.3}
$\nabla F : \inte(\mathcal{C}) \to \inte(\mathcal{C^*})$ is a bijection and $\nabla F_*$ is its inverse;\\
\item
\label{3.1.4}
$\nabla^2 F_*(-\nabla F(x))  = \left[\nabla^2 F(x)\right]^{-1},$ for all
$ x \in \inte(\mathcal{C});$\\

\noindent
$\nabla^2 F(-\nabla F_*(u))  = \left[\nabla^2 F_*(u)\right]^{-1},$ for all
 $ u \in \inte\left(\mathcal{C}^*\right)$;\\
 \item
\label{3.1.5}
$\scalPr{\nabla^2 F(x) h}{h}^{1/2} \leq \scalPr{-\nabla F(x)}{h},$ for all $ x \in \inte(\mathcal{C})$ and for all $h \in \mathcal{C}$;\\
\item
\label{3.1.6}
$\scalPr{\left(\nabla^2 F(x)\right)^{-1} h}{h}^{1/2} \leq \scalPr{x}{h},$ for all $ x \in \inte(\mathcal{C}),$ for all $ h \in \mathcal{C}^*$.\\
\end{enumerate}
\end{lemma}

\begin{proof}
The fact that every regular cone in $\mathbb{R}^n$ admits a $\vartheta$-LHSCB $F$, where $1 \leq \vartheta \leq n$, follows from Theorem~2.5.1 in~\cite{NN1994} and \cite{LM2021}. The properties \eqref{3.1.0} and \eqref{3.1.1} are essentially parts of a definition of $\vartheta$-LHSCB. For part \eqref{3.1.2}, see \cite[Proposition 2.3.4, eqn. (2.3.12)]{NN1994}. For parts \eqref{3.1.4} see \cite[eqn. (2.11)]{NT1997}. For part \eqref{3.1.3} see \cite[Proposition 5.1]{Tuncel1998}. For parts \eqref{3.1.5} and \eqref{3.1.6} see \cite[Theorem 1.1 and eqns. (1.17 and (1.18)]{NesterovT2016}, as well as \cite[Section 4]{Nesterov2018}.
\end{proof}

A regular convex cone $\mathcal{C}$ is called \emph{homogeneous} if the automorphism group of
$\mathcal{C}$ acts transitively on the interior of $\mathcal{C}$. That is, for every pair of interior points
$x,y$ of $\mathcal{C}$, there exists a nonsingular linear map $A$ such that $A(\mathcal{C})=\mathcal{C}$ and $Ax=y$. For a recent source on homogeneous cones, see~\cite{TuncelVandenberghe2022}.

A regular convex cone is called \emph{self-dual}, if there exists an inner product under which the cone is equal to its dual.

\subsection{Symmetric Cones}\label{subsec:symmetric_cones}

A regular convex cone is called \emph{symmetric} if it is homogeneous and self-dual. Let $\mathcal{C}$
be a symmetric cone. A \emph{$\vartheta$-self-scaled barrier} for $\mathcal{C}$ is a closed, at least three times continuously differentiable, strictly convex function whose
domain is the interior of $\mathcal{C}$, it is $\vartheta$-logarithmically homogeneous for $\vartheta \geq 1$, and has many additional properties
(see Definition 2.1 in \cite{NT1997} also see \cite{NT1998}). In the next lemma, we collect only those properties that we will need from self-scaled barriers in this paper.
 
\begin{lemma}
\label{lem:3.1}
Every symmetric cone $\mathcal{C} \subset \mathbb{R}^n$ admits a $\vartheta$-self-scaled barrier $F$, where $\vartheta$ is the Carath\'{e}odory number of $\mathcal{C}$.
Every $\vartheta$-self-scaled barrier $F$ for $\mathcal{C}$ is a $\vartheta$-LHSCB for $\mathcal{C}$.
Moreover, every $\vartheta$-self-scaled barrier $F$ for $\mathcal{C}$ satisfies the following:\\
\begin{enumerate}[(a)]
\item
\label{3.1.7}
$\nabla^2 F(x), \left[\nabla^2 F(x)\right]^{1/2} \in \aut(\mathcal{C})$ for all $ x \in \inte(\mathcal{C})$;\\
\item
\label{3.1.8}
for every $u \in \mathcal{C}$ the function $f_u: \inte(\mathcal{C}) \to \R$,
$f_u(x) := \scalPr{-\nabla F(x)}{u},$
is convex.\\
\end{enumerate}
\end{lemma}

\begin{proof}
It was shown in Section 3 of \cite{NT1997} that every symmetric cone $\mathcal{C}$ admits a $\vartheta$-self-scaled barrier, where $\vartheta = O(\dim(\mathcal{C}))$. Theorem 4.1 and Lemma 4.1 of \cite{GulerTuncel1998} combined with Theorem 5.5 of \cite{HG2002}, establish that every symmetric cone $\mathcal{C}$ admits a $\vartheta$-self-scaled barrier, where $\vartheta$ is the Carath\'{e}odory number of $\mathcal{C}$. 
By Definition 2.1 of \cite{NT1997}, every $\vartheta$-self-scaled barrier for $\mathcal{C}$ is a $\vartheta$-LHSCB for $\mathcal{C}$. Property \eqref{3.1.7} follows from Theorem 3.1 part (iii) of \cite{NT1997} (establishing that the Hessians yield automorphisms) and Theorem 4.1 from \cite{Tuncel1998} and Theorem 5.5 of \cite{HG2002}. Property \eqref{3.1.8} is Lemma 3.1 of \cite{NT1997}.
\end{proof}

\begin{thm}
\label{thm:symmetric_cones}
Let $\mathcal C\subseteq \R^n $ be a symmetric cone  and let $\mathcal{A}, \mathcal{B}$ be finite sets such that
$\mathcal{A}\subseteq \mathcal{C}\subseteq \R^n$, $\mathcal{B}\subseteq \mathcal{C}^\star\subseteq \R^n$ and that
$$
\scalPr{a}{b}\leq \Delta\quad\text{for all}\quad a\in \mathcal{A},\,b\in\mathcal{B}
$$ 
for some given $\Delta$, where $\mathcal{C}^* = \mathcal{C}$ under $\scalPr{\cdot}{\cdot}$. Further assume that
$\cone(\mathcal{A})$ intersects the interior of $\mathcal{C}$ and $\cone(\mathcal{B})$ intersects the interior of $\mathcal{C}^{\star}$. 
Then, there exists a self-adjoint automorphism $x\mapsto Lx$ of $\mathcal{C}$ such that  
$$
	\scalPr{La}{La}\leq \vartheta \Delta  \qquad \text{and} \qquad \scalPr{L^{-1}b}{L^{-1}b}\leq \vartheta\Delta \qquad\text{for all}\quad a\in\mathcal{A}, b\in \mathcal{B}\,,
$$
where $\vartheta$ is the Carath\'{e}odory number of $\mathcal C$.
\end{thm}

\begin{proof} Let $\mathcal C\subseteq \R^n $ be a symmetric cone and $\mathcal{A}, \mathcal{B}$ be given as above. Suppose that
both of the sets $\cone(\mathcal{A}) \cap \inte(\mathcal{C})$, $\cone(\mathcal{B}) \cap \inte(\mathcal{C}^{\star})$ are nonempty. Let $F$ be a $\vartheta$-self-scaled barrier for $\mathcal C$, where $\vartheta$ is the Carath\'{e}odory number of $\mathcal{C}$ (such $F$ exists by Lemma~\ref{lem:3.1}).  Consider the following optimization problem, where $w \in \mathbb{R}^n$ and $t \in \mathbb{R}$ are variables.
\begin{equation}\tag{\textbf{SYM1}}\label{eq:program_SYM1}
\begin{aligned}
&&\inf \quad & t\\
&&\text{subject to}\\
&&&\scalPr{-\nabla F(w)}{a}\leq t\quad &&\text{for all}\quad a\in \mathcal A\\
&&&\scalPr{w}{b}\leq t &&\text{for all} \quad b\in \mathcal B\\
&&&w\in \interior \mathcal C\,.
\end{aligned}
\end{equation}

By Lemma~\ref{lem:3.1} part \eqref{3.1.8}, the solution set of each constraint in the first group of constraints is a convex set; the second set of constraints
define a polyhedron (in fact, a polyhedral cone), the third constraint is a convex set constraint and the objective function is linear. Therefore, the optimization problem~\eqref{eq:program_SYM1} is a convex optimization problem. 
The problem is not infeasible (pick any $w$ in the interior of $\mathcal{C}$ and set $t := 1+  \max \left\{\max_{a \in \mathcal{A}} \left\{ \scalPr{-\nabla F(w)}{a} \right\}, 
\max_{b \in \mathcal{B}}\left\{ \scalPr{w}{b} \right\}\right\}$) and is not unbounded (by definition of the dual cone, $\mathcal{A}$ and $\mathcal{B}$, $t$ is nonnegative
for every feasible solution). Hence, the optimal value is finite. 

\emph{Claim:}
The optimal value of~\eqref{eq:program_SYM1}
is attained.

\emph{Proof of Claim:}
 By Lemma~\ref{lem:3.0}, part \eqref{3.1.0}, $\norm{\nabla F\left(w^{(k)}\right)} \to + \infty$,
for every sequence in the interior of $\mathcal{C}$ converging to a boundary point of $\mathcal{C}$. Since $\cone(\mathcal{A}) \cap \inte(\mathcal{C}) \neq \emptyset$, due to the first group of constraints, for every sequence in the feasible region converging to a boundary point of $\mathcal{C}$, $t \to +\infty$. Thus, the optimal value can only be attained in the interior of $\mathcal{C}$. Hence, there exists a positive integer $M$ such that the closed set $\left\{\begin{pmatrix} w \\ t \end{pmatrix} \in \inte(\mathcal{C}) \oplus \mathbb{R}_+:  \,\, F(w) \leq M \right\}$ contains all minimizers of \eqref{eq:program_SYM1}. Since, in addition, $\cone(\mathcal{B}) \cap \inte(\mathcal{C}^{\star}) \neq \emptyset$, the set
\[
\left\{\begin{pmatrix} w \\ t \end{pmatrix} \in \inte(\mathcal{C}) \oplus \mathbb{R}_+: \scalPr{w}{b} \leq \overline{t} \,\, \forall b \in \mathcal{B}, \,\,\,\, t \leq \overline{t}, \,\,\,\, F(w) \leq M \right\},
\]
where $\overline{t}$ is the optimal value, is nonempty, compact and contains all optimal solutions. By the above arguments, we proved that the convex optimization problem \eqref{eq:program_SYM1} can be written equivalently as the minimization of a linear function over a nonempty, compact set. Therefore, this latter problem as well as \eqref{eq:program_SYM1} attain their optimal value in the interior of $\mathcal{C}$.
\hfill $\Diamond$

Note that the feasible solution mentioned before the above claim is a Slater point for \eqref{eq:program_SYM1}. Thus, by necessary (and sufficient) conditions for optimality, every optimal solution $\left(\overline{w}, \overline{t}\right)$ satisfies the following conditions for some $\lambda \in \mathbb{R}^{\mathcal{A}}$, and $\mu \in \mathbb{R}^{\mathcal{B}}$.
\begin{align}
&\nabla^2 F\left( \overline w\right) \left (\sum_{a\in \mathcal{A}}\lambda_a a \right)=\left (\sum_{b\in \mathcal{B}}\mu_b b \right)\label{eq:cond1}\\
&\sum_{a\in \mathcal{A}}\lambda_a+ \sum_{b\in \mathcal{B}}\mu_b=1 \label{eq:cond2}\\
&\scalPr{-\nabla F(\overline{w})}{a}\leq \overline{t}\quad &&\text{for all}\quad a\in \mathcal A \\
&\scalPr{\overline{w}}{b}\leq \overline{t}\quad&&\text{for all} \quad b\in \mathcal B \\
&\lambda_a\geq 0,\quad \mu_b\geq 0\quad&&\text{for all} \quad a\in\mathcal A\quad\text{and}\quad b\in \mathcal B \label{eq:cond5}\\
&\lambda_a> 0\implies \scalPr{-\nabla F(\overline{w})}{a}= \overline{t}\quad&&\text{for all} \quad a\in\mathcal A\label{eq:cond6}\\
&\mu_b> 0\implies\scalPr{\overline{w}}{b}= \overline{t}\quad&&\text{for all} \quad b\in \mathcal B \label{eq:cond7}\\
&\overline{w}\in \interior \mathcal C. \label{eq:cond8}
\end{align}

Let us define $a'$ and $b'$ as follows
\begin{align*}
&a':=\sum_{a\in \mathcal{A}}\lambda_a a \qquad &&\text{and}\qquad &&b':=\sum_{b\in \mathcal{B}}\mu_b b\,.
\end{align*}

Using \eqref{eq:cond8} and Lemma~\ref{lem:3.0} part \eqref{3.1.2} together with
~\eqref{eq:cond1} we have 
\begin{align*}
	\scalPr{-\nabla F(\overline{w})}{a'}=\scalPr{\nabla^2F(\overline{w}) \overline{w}}{a'}=\scalPr{\overline{w}}{\nabla^2F(\overline{w}) a'}=\scalPr{\overline{w}}{b'}\,.
\end{align*}

By~\eqref{eq:cond5},~\eqref{eq:cond6} and~\eqref{eq:cond7}, we get
$$
\scalPr{-\nabla F(\overline{w})}{a'}= \overline{t}\sum_{a\in\mathcal A} \lambda_a
$$
and 
$$
\scalPr{\overline{w}}{b'} =\overline{t} \sum_{b\in\mathcal B} \mu_b\,,
$$
and so
$\overline{t}\sum_{a\in\mathcal A} \lambda_a=\overline{t} \sum_{b\in\mathcal B} \mu_b$ and consequently by~\eqref{eq:cond2} we have 
$$\sum_{a\in\mathcal A} \lambda_a=\sum_{b\in\mathcal B} \mu_b=1\,.$$ 

Now, we consider another optimization problem (with $w$ and $\tau$ as variables) which gets us closer to finding a scaling:
\begin{equation}\tag{\textbf{SYM2}}\label{eq:program_SYM2}
\begin{aligned}
&&\inf & \quad\tau\\
&&\text{subject to}\\
&&&\scalPr{\nabla^2 F(w)a}{a}\leq \tau\quad &&\text{for all}\quad a\in \mathcal A\\
&&&\scalPr{\left(\nabla^2 F(w)\right)^{-1}b}{b}\leq \tau \quad &&\text{for all} \quad b\in \mathcal B\\
&&&w\in \interior \mathcal C\,.
\end{aligned}
\end{equation}
By Lemma~\ref{lem:3.0} parts \eqref{3.1.5} and \eqref{3.1.6}, $\left(\overline{w}, \overline{t}^2\right)$ is
a feasible solution to \eqref{eq:program_SYM2} with objective value $\overline{t}^2$.
By Lemma~\ref{lem:3.0} part~\eqref{3.1.1}, we have  
$$
	\vartheta \scalPr{\nabla^2 F(\overline{w})a'}{a'} \geq \scalPr{-\nabla F(\overline{w})}{a'}^2= \overline{t}^2\,.
$$
Since
$$
\scalPr{\nabla^2 F(\overline{w})a'}{a'}=\scalPr{b'}{a'}\leq  \Delta\,,
$$
we have that the objective value of the feasible solution $\left(\overline{w}, \overline{t}^2\right)$ of~\eqref{eq:program_SYM2} is at most  $\vartheta\Delta$.
To finish the proof, we define  $L$ as $\left(\nabla^2 F(\overline{w})\right) ^{1/2}$ and use Lemma~\ref{lem:3.1} part~\eqref{3.1.7}.

\end{proof}

The statement of the theorem is invariant under switching $\mathcal{C}$ and $\mathcal{C}^{\star}$. Indeed, the proof of the theorem
also has this primal-dual symmetry. For self-scaled barriers, for the choice of inner-product which results in $\mathcal{C}^{\star}=\mathcal{C}$,
$F$ and $F_*$ may only differ by a constant. As a result, the derivatives of $F$ and $F_*$ are the same.

The above theorem and its proof generalize to the case that $\mathcal{A}$ and $\mathcal{B}$ are not necessarily
finite sets, but they are bounded. Next, let us consider the simpler direction (specialization) when $\mathcal{A}=\{a\}$, $\mathcal{B} =\{b\}$
(both sets are singletons). If they satisfy the hypotheses of Theorem~\ref{thm:symmetric_cones}, then $a \in \inte(\mathcal{C})$ and $b \in \inte(\mathcal{C}^{\star})$.
In this case, the ``good'' feasible solution $(\overline{w},\overline{t})$ of the optimization problem
\eqref{eq:program_SYM2}
used in the proof above is determined by the unique solution of
\[
\nabla^2 F(w) a = b, \,\,\,\, w \in \inte(\mathcal{C}).
\]
This unique solution leads to $\nabla^2 F(\overline{w})$ which is called the \emph{Nesterov--Todd scaling} in the area of interior-point methods
in convex optimization. 

Our optimization problem is trying to find the smallest possible radius $\sqrt{\tau}$ of a pair of primal-dual ellipsoids
such that
\begin{itemize}
\item
ellipsoids are centred at the origin;
\item
primal ellipsoid contains $a$, the dual ellipsoid contains $b$;
\item
the primal ellipsoids are defined by the Hessians of a self-scaled barrier for $\mathcal{C}$
(and hence in each instance, the dual ellipsoid is defined by the inverse of the specific Hessian chosen for the
primal cone, due to the definition of the dual norm).
\end{itemize}

This optimization problem is a kind of a dual problem to \emph{Todd's largest primal-dual ellipsoids problem} \cite{Todd2009}.
Todd asks `` given interior points $a$ and $b$ for $\mathcal{C}$ and $\mathcal{C}^*$ respectively, what are the \emph{largest}
primal-dual ellipsoids centred at $a$ and $b$ respectively and contained in $\mathcal{C}$ and $\mathcal{C}^{\star}$ respectively?'' Here, ``largest'' means the geometric mean of the radii of the
primal and dual ellipsoids is maximized.  So, this problem aims to construct largest inner ellipsoidal approximations to primal and dual cones simultaneously, a very useful and fundamental object
in theory and algorithms for convex optimization. Todd~\cite{Todd2009} proves that $\overline{w}$ (mentioned above) yields an optimal solution for his problem.

Let $\kappa(\mathcal{C})$ denote the Carath\'{e}odory number and $\ell(\mathcal{C})$ denote the longest chain of nonempty faces of a regular convex cone $\mathcal{C}$. Then, $\kappa(\mathcal{C}) \leq \ell(\mathcal{C}) -1$ \cite{ItoL2017}. Interestingly, in Averkov~\cite{Averkov2019} as well as in Saunderson~\cite{Saunderson2020} where extension complexity of a regular convex cone in terms of products of "low-complexity" cones is studied the geometric measure of complexity $\ell(\mathcal{C})$ of the representing cones plays a key role. In the case of symmetric cones, $\mathcal{C}$, $\kappa(\mathcal{C}) = \ell(\mathcal{C})-1$.

Our results in this subsection and the next provide bounds on the normalization constant $f_{\mathcal C}$ which depend on the barrier parameter $\vartheta$. The latter is additive under products of cones while the former is bounded by the maximum of $f_{\mathcal C}$ (see, Remark 3). Therefore, our results in this subsection and the next are strongest, when they are applied to indecomposable cones.

\subsection{Homogeneous Cones}

In the next lemma we collect the properties that we will need from certain $\vartheta$-LHSCB barriers for homogeneous cones.

\begin{lemma}
\label{lem:4.1}
Every homogeneous cone $\mathcal{C} \subset \mathbb{R}^n$ admits a $\vartheta$-logarithmically homogeneous self-concordant barrier $F$, such that
\begin{enumerate}[(a)]
\item
\label{6.1}
$\vartheta \leq n$;\\
\item
\label{6.2}
there exists a real-valued function $g$ such that
\[
F(L(x)) =F(x) + g(L), \textup{ for every } L \in \aut(\mathcal{C});\\
\]
\item
\label{6.3}
for every $u \in \mathcal{C}$ the function $f_u: \inte(\mathcal{C}) \to \R$,
$f_u(x) := \scalPr{-\nabla F(x)}{u},$
is convex;\\
\item
\label{6.4}
for every $x \in \inte(\mathcal{C})$, there exists $L \in \aut(\mathcal{C})$ such that 
$\nabla^2 F(x)= L^{\star} L$.
\end{enumerate}
\end{lemma}

\begin{proof}
For parts \ref{6.1}, \ref{6.2} and \ref{6.3} see \cite[Section 8 and Theorem 6.1]{Guler1997}. For part \ref{6.4} see \cite[Theorem 5.3]{TuncelVandenberghe2022}.
\end{proof}

\begin{thm}
Let $\mathcal C\subset \R^n $ be a homogeneous cone.
Also let $\mathcal{A}$, $\mathcal{B}$ be finite sets such that $\mathcal{A}\subseteq \mathcal{C}$, $\mathcal{B}\subseteq \mathcal{C}^\star$ and
$$
\scalPr{a}{b}\leq \Delta\quad\text{for all}\quad a\in \mathcal{A},\,b\in\mathcal{B}
$$ 
for some given $\Delta$. Further assume that
$\cone(\mathcal{A})$ intersects the interior of $\mathcal{C}$ and $\cone(\mathcal{B})$ intersects the interior of $\mathcal{C}^{\star}$. 
Then, there exists an automorphism $x\mapsto Lx$ of $\mathcal{C}$ such that  
$$
	\scalPr{La}{La}\leq \vartheta \Delta  \qquad \text{and} \qquad \scalPr{L^{-\star}b}{L^{-\star} b}\leq \vartheta\Delta \qquad\text{for all}\quad a\in\mathcal{A}, b\in \mathcal{B}\,,
$$
where $L^{-\star}$ is inverse-adjoint to $L$, and $\vartheta$ is the parameter of a $\vartheta$-LHSCB satisfying the properties listed in
the statement of Lemma~\ref{lem:4.1}.\end{thm}

\begin{proof}

Let $F$ be a $\vartheta$-LHSCB for $\mathcal C$ satisfying the hypothesis (such $F$ exists by Lemma~\ref{lem:4.1}). Again by Lemma~\ref{lem:4.1} for every point $w \in \interior \mathcal C$, $\nabla^2 F(w)$ can be represented as $L^\star L$ such that  the linear map $x \mapsto L x$ defines an automorphism of $\mathcal C$. Similarly, $\left(\nabla^2 F(W)\right) ^{-1}$ can be represented as $L^{-1} L^{-\star}$ where the linear map $x \mapsto L^{-\star} x$ defines an automorphism of $\mathcal C^\star$.

Similarly, to the proof of Theorem~\ref{thm:symmetric_cones} we can write the optimization problems in variables $(w,t)$ and $(w, \tau)$ respectively:
\begin{equation}\tag{\textbf{HOM1}}\label{eq:program_HOM1}
\begin{aligned}
&&\inf &\quad t\\
&&\text{subject to}\\
&&&\scalPr{-\nabla F(w)}{a}\leq t\quad &&\text{for all}\quad a\in \mathcal A\\
&&&\scalPr{w}{b}\leq t &&\text{for all} \quad b\in \mathcal B\\
&&&w\in \interior \mathcal C\,
\end{aligned}
\end{equation}
and 
\begin{equation}\tag{\textbf{HOM2}}\label{eq:program_HOM2}
\begin{aligned}
&&\inf & \quad\tau\\
&&\text{subject to}\\
&&&\scalPr{\nabla^2 F(w)a}{a}\leq \tau\quad &&\text{for all}\quad a\in \mathcal A\\
&&&\scalPr{\left(\nabla^2 F(w)\right)^{-1}b}{b}\leq \tau\quad &&\text{for all} \quad b\in \mathcal B\\
&&&w\in \interior \mathcal C\,.
\end{aligned}
\end{equation}

Also in the case of a homogeneous cone, the program~\eqref{eq:program_HOM2} is convex (this time by Lemma~\ref{lem:4.1}) and therefore we can apply the same analysis as in Theorem~\ref{thm:symmetric_cones} to show that the optimal objective value of ~\eqref{eq:program_HOM2} is bounded above by $\vartheta \Delta $.
\end{proof}

The following example demonstrates that the normalization through automorphisms is not applicable to some non-homogeneous  cones.

\begin{example}

Consider the cone $\mathcal C:=\{x\in \R^3\,:\, x_1^2+x_2^2\leq x_3^2\,, x_1\geq 0\,, x_3\geq 0\}$.
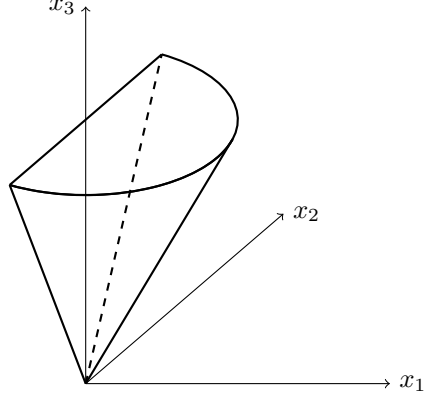
\begin{center}
	\begin{tikzpicture}
 		\draw[thick] (0,0) coordinate(A)  arc (-120:-15:2cm and 1cm) coordinate(B);
		\draw[thick]  (0,0)  arc (-120:60:2cm and 1cm) coordinate(C);
		\draw[thick]  ($0.5*(A)+0.5*(C)+(0,-3.5)$) coordinate(O)--(A);
		\draw[thick]  (O)--(B);
		\draw[thick, dashed] (O)--(C);
		\draw[thick]  (A)--(C);
		\draw[->,thin] (O)--($(4,0)+(O)$) node[right]{$x_1$};
		\draw[->,thin] (O)--($(0,5)+(O)$) node[left]{$x_3$};
		\draw[->,thin] (O)--($(O)+1.3*(C)-1.3*(A)$) node[right]{$x_2$};
	\end{tikzpicture}
\end{center}
The dual cone $\mathcal C^\star$  is not facially exposed and therefore is not homogeneous~\cite{Anh2004}. So the results in this section are not applicable to the cone $\mathcal C$. Note however that
$\mathcal{C}$ is a hyperbolic cone (intersection of two well-known hyperbolic cones: a closed half-space and a three-dimensional second-order cone) and its dual $\mathcal{C}^*$ is not~\cite[Example 4]{ChuaT2008}. 

Let us show that no appropriate rescaling is possible for $\mathcal C$ and $\mathcal C^\star$ in some cases. Automorphisms of the cone $\mathcal C$ can be parametrized by $\alpha, \beta\in \R$, $\alpha >0$ as follows
$$
x\mapsto  \alpha
\begin{bmatrix}
     1 & 0 & 0\\ 
     0 & \sqrt{\beta^2+1} & \beta\\
     0 & \beta & \sqrt{\beta^2+1} \\ 
   \end{bmatrix} x\qquad\text{and}\qquad
   x\mapsto  \alpha
\begin{bmatrix}
     1 & 0 & 0\\ 
     0 & -\sqrt{\beta^2+1} &-\beta\\
     0 & \beta & \sqrt{\beta^2+1} \\ 
   \end{bmatrix} x
$$
with the corresponding automorphisms of the dual cone $\mathcal C^\star$
$$
x\mapsto  \frac{1}{\alpha}
\begin{bmatrix}
     1 & 0 & 0\\ 
     0 & \sqrt{\beta^2+1} & -\beta\\
     0 & -\beta & \sqrt{\beta^2+1} \\ 
   \end{bmatrix} x
   \qquad\text{and}\qquad
   x\mapsto \frac{1}{\alpha}
\begin{bmatrix}
     1 & 0 & 0\\ 
     0 & -\sqrt{\beta^2+1} & -\beta\\
     0 & \beta & \sqrt{\beta^2+1} \\ 
   \end{bmatrix} x\,.
$$

Let us define sets $\mathcal A$ and $\mathcal B$ as the sets consisting of single elements  $(M, 0, M)$ and $(-M, 0, M)$, respectively, where $M$ is a large number. It is straightforward to check that no suitable rescaling of these $\mathcal A$ and $\mathcal B$ is possible through the above pairs of automorphisms. Therefore, the normalization results of this section cannot be generalized to hyperbolic cones.

\end{example}

Building upon the literature, we reference Remark 2 by Averkov et al.~\cite{Averkov2018}, which highlighted the critical role of automorphisms of the positive semidefinite cone in their proof. Our results in the preceding and current sections further demonstrate that automorphisms continue to play a significant role in more general convex cone settings.

\section{Rescaling of Cones and Extension Complexity}

This section is dedicated to the implications of factorization normalizations on the expressive power of cones. The expressive power of a cone $\mathcal C\subseteq \mathcal R^n$ will be reflected in the pairs $\mathcal V$, $\mathcal F$  that admit a factorization over the cone $\mathcal C$, where  $\langle v, f\rangle\geq 0$ for every $v\in \mathcal V$ and $f\in \mathcal F$. In other words, we are interested in the pairs $\mathcal V$, $\mathcal F$ such that  for every $v\in \mathcal V$ and for every $f\in \mathcal F$ there are $a_v\in \mathcal C$ and $b_f \in \mathcal C^\star$ with $\langle v, f\rangle=\langle a_v, b_f\rangle $ for every $v\in \mathcal V$ and for every $f\in \mathcal F$.

The intuition behind the study of limitations in the expressive power of a cone $\mathcal C$ is as follows. First, we consider a specific collection of pairs  $\mathcal V$, $\mathcal F$, where  $\langle v, f\rangle\geq 0$ for every $v\in \mathcal V$ and $f\in \mathcal F$. We plan to prove that some of the pairs in the collection cannot be factorized using the cone $\mathcal C$. Second, for the sake of argument, let us consider a pair $\mathcal V$, $\mathcal F$ that can be factorized using the cone $\mathcal C$. Consider the sets 
\[
\mathcal A:=\left\{a_v\,:\, v\in \mathcal V\right\}\qquad \text{and}\qquad \mathcal B:=\left\{b_f\,:\, f\in \mathcal F\right\}\,.
\]
We show that there exists a set $\Gamma$ (which does not depend on $\mathcal V$ and $\mathcal F$)  and there exist at most $2n$ vectors in $\Gamma$, such that  the whole pair $\mathcal V$ and $\mathcal F$ can be identified from these $2n$ vectors in $\Gamma$. The $2n$ vectors in $\Gamma$ are obtained by carefully selecting a set of at most $2n$ vectors from $\mathcal B$ and then for each of them picking a ``closest"  vector from $\Gamma$.
Third, the cardinality of the constructed set $\Gamma$ depends only on the cone $\mathcal C$, in particular, it depends only on the dimension of its ambient space $n$ and the normalization constant $f_{\mathcal C}$. Now using the counting argument, for a sufficiently large collection of pairs $\mathcal V$, $\mathcal F$ we would not be able to identify each of them using at most $2n$ vectors from $\Gamma$. This leads to the statement that the cone $\mathcal C$ does not have enough of expressive power to capture all of the studied pairs $\mathcal V$, $\mathcal F$.

\subsection{Relationship to Polytopes}

Let us now consider how the pairs $\mathcal V$, $\mathcal F$ appear in the context of polytopes. We assume that a set $\mathcal V$ is a subset of $\mathcal P\subseteq \mathcal R^d$ and $\mathcal F$ is a subset of $\mathcal H\subseteq\mathcal R^d$. For example, $\mathcal P$ can correspond to the set of all $0/1$ points in $\mathcal R^{d-1}$ and $\mathcal H$ can correspond to the set of all possible hyperplanes spanned by $d-1$ affinely independent $0/1$ points in $\mathcal R^{d-1}$. 
The next theorem corresponds to representing  the convex hull of $\mathcal{X}\subseteq \{0,1\}^d$ through its outer and inner descriptions, i.e. once as a set of linear inequalities, say $\mathcal F$, and as a set of its extreme points, say $\mathcal V$. The upper bound on $M$ in the below theorem can be obtained for example through a straightforward application of Cramer's rule.

\begin{thm}\label{thm:01polytopes}
Given a nonnegative integer $d$, there is a way to assign to each $\mathcal{X}\subseteq \{0,1\}^{d-1}$ a unique $\mathcal V \subseteq \mathcal \{0,1\}^{d-1}\times\{1\}$ and a unique $\mathcal F\subseteq \Z^{d}$ such that
\begin{enumerate}[(i)]
\item $\|f\|_\infty\leq M$, $\|v\|_\infty\leq 1$ for all  $v\in {\mathcal V}$ and $f\in {\mathcal F}$;
\item $\langle v,\, f\rangle \geq 0$ for all  $v\in {\mathcal V}$ and $f\in {\mathcal F}$;
\item for all $v\in \{0,1\}^{d-1}\times\{1\}$, $v\notin\mathcal V$ there exists $f\in \mathcal F$  such that $\langle v,\, f\rangle \leq -1$,
\end{enumerate}
where $M=2^{d\log(2d)}$.
\end{thm}

Similarly, the next theorem corresponds to outer and inner descriptions of cyclic polytopes. We can use Lemma 5 in~\cite{Fiorini12} to estimate the coefficients in an outer description of a cyclic polytope, i.e., to obtain an upper bound on $M$.

\begin{thm}\label{thm:cyclic}
Given a nonnegative integer $d$, there is a way to assign to each  subset $\mathcal{X}$ of the set $ \left\{(k,k^2,\ldots, k^{d-1})\,:\, k\in \{0, \ldots,t\}\right\}$ a unique $\mathcal V \subseteq \mathcal \Z^{d-1} \times \{1\}$ and a unique $\mathcal F\subseteq \Z^{d}$ such that
\begin{enumerate}[(i)]
\item $\|f\|_\infty\leq M$, $\|v\|_\infty\leq t^{d-1}$ for all  $v\in {\mathcal V}$ and $f\in {\mathcal F}$;
\item $\langle v,\, f\rangle \geq 0$ for all  $v\in {\mathcal V}$ and $f\in {\mathcal F}$;
\item for all $v\in \mathcal \Z^{d-1}\times\{1\}$, $v\notin\mathcal V$ there exists $f\in \mathcal F$  such that $\langle v,\, f\rangle \leq -1$,
\end{enumerate}
where $M=((d+1)t^d)^d$.
\end{thm}

\subsection{Compact Encoding }
In this section, we show Lemma~\ref{rounding} that guarantees that under some conditions, the pairs $\mathcal V$ and $\mathcal F$ can be compactly encoded. Before stating and proving Lemma~\ref{rounding} we need some technical result about ``nets". The next lemma gives us a way to round the vectors in one of the factorization parts such that the rounded vectors lie in a set $\Gamma$, where $\Gamma$ depends only on the cone $\mathcal C$ and the set $\Gamma$ has an appropriate bound on its cardinality. Later, $\rho$ in the lemma will be a bound on the norms of the vectors in the factorization over the cone $\mathcal C$ and its dual and; $\epsilon$ will indicate the precision of rounding.  The lemma below follows from a classical result of Rogers~\cite{Rogers1963}, also see~\cite{Verger-Gaugry2005}. The lemma is used for going from exact factorizations to ``approximate" factorizations, where one of the factors can use a bounded number of vectors. 

\begin{lemma}\label{epsilon-net}
    Fix a dimension $n\geq 3$, radius $\rho > 0$, and $\epsilon > 0$ such that  $\rho/\epsilon \geq n$. Then there exists a set $ \Gamma:= \Gamma(n,\rho,\epsilon) \subseteq \R^n$ such that
   \begin{enumerate}[(i)]
        \item \label{epsilon-net.1} for all $u \in B_0(\rho):=\{x\in \R^n\,:\, \|x\|_2\leq \rho\}$ there is $\ov{u} \in \Gamma$ such that $\|u - \ov u\|_2 \leq \epsilon$, and
        \item \label{epsilon-net.2} $
|\Gamma| \leq \exp(1)(n\ln n+n \ln\ln n +5 n) \left({\rho}/{\epsilon}\right)^n
$.
    \end{enumerate}
\end{lemma}

The next lemma tells that under some conditions on $\mathcal V$, $\mathcal F$, the set $\mathcal V$ can be reconstructed using at most $2n$ vectors from a sufficiently small $\Gamma$, i.e. $\mathcal V$ can be reconstructed  from a relatively small amount of information. 

\begin{lemma}\label{rounding}
Let  $\mathcal C\subseteq \R^n$  be a regular convex cone that can be normalized with coefficient $f_{\mathcal C}$.  Let us have two families of vectors $\mathcal P$,  $\mathcal H\subseteq \R^d$, nonnegative integer $M$, and sets $\mathcal V \subseteq \mathcal P$, $\mathcal F\subseteq \mathcal H$  such that
\begin{enumerate}[(i)]
\item\label{item:F-collection1} $\|f\|_\infty\leq M$, $\|v\|_\infty\leq 1$ for all  $f\in {\mathcal F}$ and $v\in {\mathcal V}$;
\item \label{item:F-collection2}$\langle v,\, f\rangle \geq 0$ for all  $v\in {\mathcal V}$ and $f\in {\mathcal F}$;
\item \label{item:F-collection3} for all $v\in \mathcal P$, if $v\notin\mathcal V$ then there exists $f^*\in \mathcal F$  such that $\langle v,\, f^*\rangle \leq -1$;
\item \label{item:F-collection4} the pair $\mathcal V$, $\mathcal F$ can be factorized using the cone $\mathcal C$.
\end{enumerate}
Let us define $\rho := \sqrt{(d+1)M} \cdot f_{\mathcal C}$ and $\epsilon: =\left(4(n+d) \rho \right)^{-1}$. Then for every  $i \in \{1,\ldots, n+d\}$ there exists
    \begin{enumerate}[(1)]
        \item a vector $f_i \in \mathcal{H}$ such that $\|f_i\|_\infty \leq M$,
                \item a vector $u_{f_i} \in \Gamma := \Gamma\left(n , \rho, \epsilon \right)$, where $\Gamma$ is as in Lemma~\ref{epsilon-net};
    \end{enumerate}
    such that the following holds
    \begin{align*} \mathcal{V} = \{ x \in \mathcal P \,: \,\text{there exists } &y \in \mathcal{C^*} \cap B_0(\rho) \,\text{ such that }\\ \, &\left\vert \inprod{f_i, x}- \inprod{u_{f_i}, y} \right\vert \leq \frac{1}{4(n+d)} \quad\text{for all }\, i \in \{1,\ldots, n+d\}\}. \end{align*}
\end{lemma}

\begin{proof} By assumption there is a $\mathcal C$-factorization $(a_v)_{v \in \mathcal{V}} \subseteq \mathcal C$, $(b_f)_{f \in {\mathcal F}}\subseteq \mathcal C^*$ for $\mathcal V$ and $\mathcal F$ such that for every $v \in \mathcal V$ and every $f\in \mathcal F$ we have
    \[  \inprod{v, f} = \inprod{a_v, b_f}\,. \]
    Due to normalization of the cone $\mathcal C$, we may assume that $\|a_v\|_2 \leq \sqrt{d M} \cdot f_{\mathcal{C}}$ for all $v\in \mathcal{V}$ and that $\|b_f\|_2 \leq \sqrt{d M}\cdot f_{\mathcal{C}}$ for all $f \in \mathcal{F}$.
    Let us now select a ``subsystem of maximum volume", as in \cite{rothvoss201101} and \cite{briet2013existence}. For linearly independent set of vectors $x_1,\ldots, x_k \subseteq \R^n$ denote the $k$-dimensional parallelepiped volume
    \[ \vol\left(\sum_{i=1}^k  \lambda_i x_i \,:\, \lambda_i \in [0,1] \quad\text{for all }\,i\right) = \det\left((x_i^{\top} x_j)_{ij}\right)^{1/2}. \]
    Clearly, if the vectors   $x_1,\ldots, x_k \subseteq \R^n$ are dependent the volume is zero. Let us define 
    \[W: = \linspan \left(\{(f,b_f): f \in \mathcal{F}\}\right)\,,\] and let $I' \subseteq \mathcal{F}$ be a subset of size $|I'| = \dim(W)$ such that $\vol(\{(f,b_f): f \in I'\})$ is maximized. Note that $|I'| \leq n + d$ since the dimension of the linear space $W$ is at most $n+d$.
    
    Now, for each $f \in \mathcal{F}$ let $\ov{b}_f$ be an element of $\Gamma$ such that 
    \[\|b_f - \ov{b}_f\|_2 \leq \epsilon = \left(4(n+d)\sqrt{(d+1)M} \cdot f_{\mathcal C}\right)^{-1},\]
    which exists by property~\eqref{epsilon-net.1} of $\Gamma$. Let us define the set $\ov{\mathcal{V}} $ as the following set
    \begin{align*}\ov{\mathcal{V}} :=\{ v\in \mathcal P \,:\, &\text{there exists } y \in\mathcal{C} \cap B_0(\rho)\text{ such that }  \\
    &\left\vert \inprod{v, f} - \inprod{\ov{b}_f, y} \right\vert \leq \frac{1}{4(n+d)}\, \text{for all }f \in I' \}. \end{align*}
    In the remaining part of the proof we show that $\ov{\mathcal{V}} = \mathcal{V}$ as desired in the statement of the lemma, by selecting $f_i$, $i \in \{1,\ldots,n+d\}$ to be the vectors in $I'$ (duplicating them, if needed) and   setting $u_f:=\ov{b}_f$ for $f \in I'$.
    
    First, we show that $\mathcal{V} \subseteq \ov{\mathcal{V}}$. Let us consider some  $v $ in $\mathcal{V} $, and then let us show that $v$ is also in $\ov{\mathcal{V}}$.  For every $v\in \mathcal{V}$ we have $\|a_v\|_2 \leq\rho = \sqrt{(n+1)M}\cdot  f_{\mathcal{C}}$, so for every $f\in I'$ we have
    \begin{align*}
        |\inprod{v, f} - \inprod{\ov{b}_f,a_v}|
        &= |\inprod{v, f} - \inprod{a_v,b_f} + \inprod{b_f - \ov{b}_f,a_v}|=   |\inprod{b_f - \ov{b}_f,a_v}|\\
        &\leq \|b_f - \ov{b}_f\|_2 \|a_v\|_2 \leq \frac{1}{4(n+d)},
    \end{align*}
    where the last line follows from the Cauchy-Schwarz inequality. Thus, by the definition of $\ov{\mathcal{V}}$, we conclude that $v$ lies in $\ov{\mathcal{V}}$.
   
    Now we show the reverse inclusion that  $  \ov{\mathcal{V}}\subseteq \mathcal{V}$. For this, suppose that there exists  $v\in \mathcal P$ that is not in $\mathcal{V}$ and let us show that $v$ is also not in  $\ov{\mathcal{V}}$. Since  $v$ is not in $\mathcal{V}$, by~\eqref{item:F-collection3} we know that there exists some $f^* \in \mathcal{F}$ such that $\inprod{f^*,v} \leq -1$. 
    
   Now, it may not be the case that $f^*$ is in $I'$. However, the vector $(f^*, b_{f^*})$ is in the linear space $W$, and so we can express $(f^*, b_{f^*})$ as a linear combination of the vectors $(f, b_f)$, $f\in I'$. In particular, there are unique multipliers $\nu \in \R^{I'}$ such that $(f^*, b_{f^*}) = \sum_{f \in I'} \nu_i(f , b_f)$.     
    Using the choice of subsystem $I'$ as a system with a maximum volume and using Cramer's rule, we have that for every $f'\in I'$
    \[ |\nu_{f'}| = \frac{\vol(\{(f,b_f): f \in I'\del\{f'\} \cup \{f^*\} \})}{\vol(\{(f,b_f): f \in I'\})} \leq 1\,.\]
    For every $y \in \mathcal{C} \cap B_0(\rho)$, we have  $\inprod{b_{f^*}, y} \geq 0$  and so we also have
    \begin{align*}
        1
        &\leq \left\vert -\inprod{f^*,v} + \inprod{b_{f^*}, y}\right\vert 
        = \left\vert \sum_{f\in I'} \nu_f \left( -\inprod{f,v} + \inprod{b_f, y} \right) \right\vert\\
        &\leq \sum_{f\in I'} |\nu_f | \left\vert -\inprod{f,v} + \inprod{b_f, y} \right\vert
        \leq (n+d) \max_{f \in I'} \left\vert -\inprod{f,v} + \inprod{b_f, y}\right\vert\,.
    \end{align*}
    Also for every $y \in \mathcal{C} \cap B_0(\rho)$ and  for every $f \in I'$, we have
    \begin{align*}
        \left\vert -\inprod{f,v} + \inprod{b_f, y} \right\vert
        &= \left\vert -\inprod{f,v} + \inprod{\ov{b}_f, y}+ \inprod{b_f- \ov{b}_f, y} \right\vert \\
        &\leq \left\vert -\inprod{f,v} + \inprod{\ov{b}_f, y}\right\vert + \left\vert \inprod{b_f- \ov{b}_f, y}  \right\vert\\
        &\leq \left\vert -\inprod{f,v} + \inprod{\ov{b}_f, y}\right\vert + \frac{1}{4(n+d)}\,,
    \end{align*}
    where the last inequality follows from $\| b_f- \ov{b}_f \|_2\leq \epsilon = \left(4(n+d)\rho \right)^{-1}$ and $\| y \|_2\leq \rho$.
    Combining this with the previous inequality \[\max_{f \in I'} \left\vert -\inprod{f,v} + \inprod{b_f, y}\right\vert\geq 1/(n+d)\,,\] we get
    \begin{align*}\max_{f \in I'} \left\vert -\inprod{f,v} + \inprod{\ov{b}_f, y}\right\vert\geq \max_{f \in I'} \left\vert -\inprod{f,v} + \inprod{b_f, y}\right\vert  - \frac{1}{4(n+d)}\\\geq\frac{1}{(n+d)} - \frac{1}{4(n+d)} \geq \frac{1}{2(n+d)}\end{align*}
    which implies that $ v$ does not lie in $\ov{\mathcal{V}}$.
\end{proof}

The next corollary outlines the strategy that is used for utilizing Lemma~\ref{rounding} for establishing lower bounds for extension complexity. 
\begin{cor}\label{cor_rounding}
Let  $\mathcal C\subseteq \R^n$  be a regular convex cone that can be normalized with coefficient $f_{\mathcal C}$. Let us have two families of vectors $\mathcal P$,  $\mathcal H\subseteq \R^d$ and a nonnegative integer $M$. Let us assume that we have a collection of $N$ objects each of which leads to a unique $\mathcal V  \subseteq \mathcal P$ and a unique $\mathcal F\subseteq \mathcal H$ such that~\eqref{item:F-collection1},~\eqref{item:F-collection2},~\eqref{item:F-collection3},~\eqref{item:F-collection4} in Lemma~\ref{rounding} are satisfied. Let us define $\rho := \sqrt{(n+1)M }\cdot f_{\mathcal C}$ and $\epsilon :=\left(4(n+d) \rho \right)^{-1}$. 

If every above pair $\mathcal V$, $\mathcal F$ can be factorized using the cone $\mathcal C$ then we have \[|N|\,\leq\,  \left(\Gamma(n,\rho,\epsilon) \,\cdot\,|\{f\in \mathcal{H}\,:\,\|f\|_\infty \leq M \}|\right)^{n+d}\,.\]
\end{cor}
\begin{proof}
    Indeed, by Lemma~\ref{rounding} each $\mathcal V$ can be uniquely identified by $n+d$ vectors in $\{f\in \mathcal{H}\,:\,\|f\|_\infty \leq M \}\times \Gamma(n,\rho,\epsilon)$, leading to the desired inequality.
\end{proof}

\subsection{Relationship to Extension Complexity}

Using Corollary~\ref{cor_rounding} and Theorem~\ref{thm:01polytopes}, we can show that there are 0/1 polytopes in $\R^d$ with no  extension over  a cone $\mathcal{C}\subseteq \R^n$ if both $f_{\mathcal C}$ and $n$ are polynomial in $n$. The statement is equivalent to showing that at least one of the pairs $\mathcal V$, $\mathcal F$ from Theorem~\ref{thm:01polytopes} cannot be factorized using the cone $\mathcal C$. The proof of this fact follows the proofs of \cite{rothvoss201101} and \cite{briet2013existence}.  
For the sake of contradiction, let us assume that every  pair $\mathcal V$, $\mathcal F$ as in Theorem~\ref{thm:01polytopes} can be factorized using a cone $\mathcal C$ for each $\mathcal{X}\subseteq \{0,1\}^d$, where $\mathcal C\subseteq \R^n$ can be normalized with coefficient $f_{\mathcal C}$. Thus by Lemma~\ref{epsilon-net} and Lemma~\ref{rounding}, the total number of different possible elements $\Gamma$ is at most
\begin{align*}
|\Gamma| \leq &\exp(1)(n\ln n+n \ln\ln n +5 n) \left(\frac{\rho}{\epsilon}\right)^n =\\
&\exp(1)(n\ln n+n \ln\ln n +5 n)\left(4\cdot \rho^2 \cdot(n+d)\right)^n\leq 
\\
&\exp(1)(n\ln n+n \ln\ln n +5 n)\left( 4\cdot (d+1) M \cdot f^2_{\mathcal C}\cdot (n+d) \right)^n\leq \\
&n^2\cdot \left( 3 \cdot (d+1) M \cdot f^2_{\mathcal C} \cdot n \right)^n\,.\end{align*}
Thus, the total number of different possible $\mathcal V$ is at most
$$\left(\left(2\cdot 2^{d\log(2d)}+1\right)^{d+1} \cdot n^2\cdot \left( 3 \cdot (d+1) 2^{d\log(2d)} \cdot f^2_{\mathcal C} \cdot n \right)^n\right)^{n+d}$$
and so 
$$ 2^{2^d}-1\leq \left(\left(2\cdot 2^{d\log(2d)}+1\right)^{d+1}\cdot n^2\cdot \left( 3 \cdot (d+1) 2^{d\log(2d)} \cdot f^2_{\mathcal C} \cdot n \right)^n\right)^{n+d},$$
showing that $n\cdot f_{\mathcal C}$  has to be exponentially large with respect to $d$ in order for $\mathcal C$ to lead to an extension of every $0/1$ polytope of dimension $d$.

Similarly, Corollary~\ref{cor_rounding} and Theorem~\ref{thm:cyclic} show that there are cyclic polytopes corresponding to Theorem~\ref{thm:cyclic} with no  extension over a cone $\mathcal{C}\subseteq \R^n$ if both $f_{\mathcal C}$ and $n$ are polynomial in $n$ and $t$. Lower bounds for cyclic polytopes were studied in~\cite{Fiorini12} and this proof would be analogous to theirs.

\bibliography{literature.bib}

\end{document}